\documentclass[12pt]{amsart}
\usepackage{geometry}   

\usepackage[colorlinks,citecolor = red, linkcolor=blue,hyperindex]{hyperref}
\usepackage{euscript,eufrak,verbatim, mathrsfs}
\usepackage[psamsfonts]{amssymb}
\usepackage{bbm}
\usepackage{graphicx}
\usepackage{float}
\usepackage{float, tikz}

\usepackage{extpfeil}

\usepackage[all, cmtip]{xy}

\usepackage{upref, xcolor, dsfont}
\usepackage{amsfonts,amsmath,amstext,amsbsy, amsopn,amsthm}
\usepackage{enumerate}

\usepackage{url}

\usepackage{mathtools}
\usepackage{bookmark}

\usepackage{euscript}
\usepackage{times}
\usepackage{type1cm}        
\usepackage{multicol}        
\usepackage[bottom]{footmisc}
\newcommand{\normmm}[1]{{\left\vert\kern-0.25ex\left\vert\kern-0.25ex\left\vert #1 
		\right\vert\kern-0.25ex\right\vert\kern-0.25ex\right\vert}}

\newtheorem{theorem}{Theorem}[section]
\newtheorem*{theorem*}{Theorem B}
\newtheorem{lemma}[theorem]{Lemma}

\newtheorem{proposition}[theorem]{Proposition}

\newtheorem{corollary}[theorem]{Corollary}

\newtheorem*{observation*}{Observation}

\newtheorem*{assumption*}{Assumption}
\newtheorem*{question*}{Question}

\theoremstyle{definition}

\newtheorem*{definition*}{Definition}

\theoremstyle{remark}

\newtheorem*{remark*}{Remark}

\numberwithin{equation}{section}

\begin{document}

	\author{Mingjie Tan}
    
	\begin{abstract}
We prove that the non-covered set in Dvortezky random covering is a set of multiplicity, by showing that the natural multiplicative chaotic measure is a Rajchman measure.
	\end{abstract}

	
	\title{The non-covered set in Dvoretzky covering is a set of multiplicity}
	\maketitle

	\setcounter{tocdepth}{2}

	\setcounter{tocdepth}{0}
	\setcounter{equation}{0}



\section{Introduction}

Let us recall the Dvoretsky random covering. Let $\{\omega_n\}_{n \geq 1}$ be a sequence of independent random variables on a probability space $(\Omega, \mathcal{B}, \mathbb{P})$ which are uniformly distributed over the circle $\mathbb{T} := \mathbb{R}/\mathbb{Z}$ identified with the interval $[0, 1)$. Let $\{\ell_n\}_{n \geq 1}$ be a sequence of positive real numbers which is decreasing to zero. For every $n \geq 1$, consider the random interval $\mathcal I_n := (\omega_n, \omega_n + \ell_n) \mod 1$. 
The Dvoretzky covering problem arised by Dvoretzky in \cite{Dvo56} is to determine the conditions on $\{\ell_n\}_{n\geq1}$ such that $\mathbb T$ is covered infinitely many times almost surely, namely,
\begin{equation*}
    \mathbb{P}\left(\mathbb{T} = \limsup_{n \to \infty} I_n \right) = 1. \label{D-problem}
\end{equation*}

This problem attracts many attentions \cite{Kah59, Erd61, Bil65}. L. A. Shepp \cite{SH72} finally found the following necessary and sufficient condition for $\mathbb T$ to be covered:
\begin{equation}
    \sum_{n=1}^{\infty} \frac{1}{n^2} \exp \left( \ell_1 + \cdots + \ell_n \right) = \infty. \label{S}
\end{equation}

When Sheep's condition \eqref{S} does not hold, the finitely covered set 
$$
F:=\mathbb{T}\setminus\limsup_{n\rightarrow\infty}I_n
$$
was also studied in \cite{Kah85b}. 
In particular, in the case where $\ell_n=\frac{\alpha}{n}$, $\alpha\in(0,1)$, Kahane calculated the Hausdorff dimension $\dim_HF=1-\alpha$.

 We consider the non-covered set 
 $$
 E :=\mathbb{T}\setminus \left(\bigcup_{n=1}^\infty I_n\right).
 $$
 This is a closed set in $\mathbb{T}$, a subset of $F$. We will study it from the point of view of harmonic analysis.

A subset $S$ of the circle $\mathbb{T}$ is called a \emph{set of uniqueness }($U$-set) if any trigonometric series $\sum c_n e^{i n x}$ that converges to zero outside $S$ must have all coefficients $c_n = 0$. Otherwise, it is called a \emph{set of multiplicity }($M$-set). See \cite{KL95} and \cite{Lyo20} for details.

It is proved that a closed subset \( S \subset \mathbb{T} \) is an $M$-set if and only if there exists a non-zero distribution supported on \( S \) whose Fourier coefficients converge to 0 as \( |n| \to \infty \) (\cite[p.70]{KL95}). If there exists a non-zero measure supported on \( S \) whose Fourier coefficients converge to $0$ as \( |n| \to \infty \) (such a measure is called Rajchman measure) then \( S \) is called an \( M_0 \)-set.

 In this paper, we can prove that the non-covered set $E$ is, almost surely, a $M_0$-set when $\ell_n=\frac{\alpha}{n}$, for some $0<\alpha<1$. A result for more general $\ell_n$ will be stated in Theorem~\ref{multiple_convolution_result} below. In particular, if $0<\alpha\leq\frac{1}{2}$, we can take $d=1$, in this case, $\mu_D$ is a Rajchman measure.

To prove $E$ is a $M_0$-set, e will consider the following natural measure on $E$ (see \cite{Kah85b}). 
For any $t\in\mathbb [0,1)$ and $n\geq1$,
    let 
    $$
    P_n(t) := \frac{1-\textbf{1}_{(0,\ell_n)}(t-\omega_n)}{1-\ell_n},\quad M_n(t)\coloneqq \prod_{k=1}^nP_k(t), \quad \mu_n(\mathrm{d}t):=M_n(t)\mathrm{d}t.
    $$
The measure $\mu_n$ is absolutely continuous with respect to the Lebesgue measure $\mathrm{d}t$ with density $M_n(t)$. It is shown in \cite{Kah85b} that the random measure $\mu_n$, almost surely, converges weakly to a random measure
\begin{equation}
\mu_{D}(\mathrm{d}t) 
:= \lim_{n\rightarrow\infty} \mu_n(\mathrm{d}t) \quad \text{a.s.}  \label{def_random_coverings}
\end{equation}
We call $\mu_D$ the Dvoretzky measure. 

Our result will be proved under the following two assumptions.
\begin{align}
&\sum_{n=1}^{\infty} \ell_n = \infty, \tag{A}\label{A} \\
&\sum_{n=1}^{\infty} (\ell_n - \ell_{n+1}) \exp\left( \sum_{k=1}^{n} \ell_k \right) < \infty. \tag{B}\label{B}
\end{align}
Assumption \eqref{A} and Borel-Cantelli lemma imply that, almost surely, the set \(E\) has zero Lebesgue measure. Assumption \eqref{B} ensures that \(E\) is almost surely non-empty by Shepp's theorem mentioned above, as condition \eqref{B} violates Shepp's condition. The condition \eqref{B} also implies that $\sum_{n=1}^\infty\ell_n^2<\infty$ (cf. \cite[p.145]{Kah85b}).

The function $K(t)\coloneqq \exp{(\sum_{n=1}^\infty(\ell_n-|t|)_+)}$, for $|t|<\frac{1}{2}$, plays a key role in the study of Dvoretzky random covering. Our main result is stated as follows.

\begin{theorem}
 Suppose that the sequence $(\ell_n)$ satisfies the assumptions \eqref{A}, \eqref{B} and there exists an integer $d\geq1$ such that 
\begin{equation}
    \int_{\mathbb{T}^d}K(t_1)\cdots K(t_d) K(-\sum_{i=1}^dt_i)\mathrm{d}t_1\cdots \mathrm{d}t_d<\infty.
    \label{L1condition_RC}
\end{equation}
 Then the d-th convolution $\ast^d\mu_{D}$ of $\mu_D$ is a.s. absolutely continuous.
\label{multiple_convolution_result}
\end{theorem}
Theorem~\ref{multiple_convolution_result}, together with the Riemann-Lebesgue lemma, implies immediately the following corollary.
\begin{corollary}
    Under the assumptions of Theorem~\ref{multiple_convolution_result}, $\mu_{D}$ is a Rajchman measure and the non-covered set $E$ is a $M_0$-set.\label{cor1}
    \end{corollary}

    A canonical example is $\ell_n=\frac{\alpha}{n}$, $\alpha\in(0,1)$. In this case, Corollary~\ref{cor1} says that $E$ is almost surely a $M_0$-set, because if we take $d\geq \frac{\alpha}{1-\alpha}$, \eqref{L1condition_RC} is satisfied (cf. \cite[Lemma~2.3]{GV24}). In particular, if $0<\alpha<\frac{1}{2}$, we can take $d=1$, and in this case, $\mu_D$ is a Rajchman measure.

Garban and Vargas, in their remarkable work of \cite{GV24}, explored the Rajchman property of another type of multiplicative chaos measure, specifically the Gaussian Multiplicative Chaos measure, which can be formally understood as:
\begin{equation*}
\label{eq:GMC}
M_{GMC,\gamma}(\mathrm{d}t) := e^{\gamma \varphi(t) - \frac{\gamma^2}{2} \mathbb{E}[\varphi(t)^2]} \mathrm{d}t
\end{equation*}
where  \(\gamma \in (0,\sqrt{2})\) and \(\varphi\) is the Gaussian Free Field on \(\mathbb{T}\) with covariance kernel
\[
\mathbb{E}[\varphi(t) \varphi(t^\prime)] = \ln \frac{1}{|e^{2\pi i t} - e^{2\pi i t^\prime}|}.
\]
See \cite{Kah85a} and \cite{Ber17} for its detailed construction.
Due to the inherent complexity in the definition of the Gaussian Multiplicative Chaos $\mu_{GMC,\gamma}$, estimating the Fourier coefficients $\widehat{\mu}_{GMC,\gamma}(n)$ poses significant challenges. Garban and Vargas addressed this issue by considering multiple convolutions of the $\mu_{GMC,\gamma}$, transforming it into an integrable function. 
We will follow their method. 

\section{proof of Theorem~\ref{multiple_convolution_result}}
The basic idea is the following. Assume that a sequence of measures $\nu_n$ converges weakly to $\nu_\infty$ and $\nu_n$ admits its density $f_n$ with respect to Lebesgue measure. If some subsequence $\left(f_{n_k}\right)$ converges in $L^1-$norm, then the limit measure $\nu_\infty$ is absolutely continuous. Apply this idea to the convolution of $\mu_n$ could allow us to show that $\mu_\infty$ is Rajchman.
The following lemma deals with a sequence of random density functions such that $f_n(\omega,\cdot)\in L^1(\mathbb{T})$ almost surely.

\begin{lemma}
    Let $(f_n(\omega,t))_{n\geq1}$ be a sequence of random functions. Suppose that there exists $\alpha \in (0,1)$ such that 
    \begin{equation}
        \underset{n,m\rightarrow \infty}{\lim}\mathbb E \left[\|f_n(\omega,\cdot) -f_m(\omega, \cdot)\|_{1}^\alpha\right] = 0. \label{alpha}
    \end{equation}
    Then there exists a subsequence  $\{f_{n_k}\}_{k\geq 0}$ that converges in $L^1(\mathbb T)$ almost surely.
    \label{core_lemma}
\end{lemma}
\begin{proof}
    The hypothesis \eqref{alpha} implies that there is a subsequence $\{f_{n_k}\}_{k\geq1}$ such that 
    \begin{equation*}
         \forall k \geq1, \qquad \mathbb E \left[\|f_{n_k}(\omega,\cdot) -f_{n_{k+1}}(\omega,\cdot)\|_{1}^\alpha)\right] \leq \frac{1}{3^k}.
    \end{equation*}
By the Markov inequality, for any $k\geq1$,
\[
\mathbb{P}\left[\|f_{n_k}-f_{n_{k+1}}\|_1>\frac{1}{2^{\frac{k}{\alpha}}}\right]\leq(\frac{2}{3})^k.
\]
    Then Borel-Cantelli Lemma implies that $\sum_k\| f_{n_k} - f_{n_{k+1}} \|_{1} < \infty$ almost surely.
 \end{proof}
 
 In order to show Theorem~\ref{multiple_convolution_result}, by Lemma~\ref{core_lemma}, it suffices to show
\begin{equation}
        \underset{n,m\rightarrow \infty}{\lim}\mathbb E \left[\|\ast^d M_n(t) -\ast^d M_m(t)\|_1^\frac{1}{2}\right] = 0. \label{ultimate_goal}
    \end{equation}
where $\ast^d M_n(t)$ is the $d$-fold convolution of the density $M_n(t)$ of $\mu_n$. It is clear that 
\begin{equation}
\ast^d M_n(t) = \int_{\mathbb{T}^d} M_n(t_1)\cdots M_n(t_d)M_n(t -\sum_{i=1}^dt_i) \mathrm{d}t_1\mathrm{d}t_2\cdots \mathrm{d}t_d. \label{def-convolution}
\end{equation}

For simplicity, the integrand in \eqref{def-convolution} will be denoted
\begin{equation}
     F_n(\hat{t},t) :=  M_n(t_1)\cdots M_n(t_d)M_n(t_{d+1})\label{def_F_n}
\end{equation}
where 
$$
\hat{t}=\left(t_1,\cdots,t_d\right)\quad \text{and}\quad t_{d+1} = t -\sum_{i=1}^dt_i.
$$

To facilitate calculations, we consider the distance on \(\mathbb{T}\) as \(\|t\| := \inf_{n \in \mathbb{Z}} |t - n|\) for any $t\in\mathbb{T}$. Note that when $|t|\leq\frac{1}{2}$, we have $|t|=\|t\|$.

We will need the following well-known estimation.
\begin{lemma} Let $K_n(t)\coloneqq \exp{(\sum_{k=1}^n(\ell_k-|t|)_+)}$ for any $t\in\mathbb{T}$. Then
there is $C> 0$ such that for any $t,t^\prime \in \mathbb T$ and any $n\geq1$, one has
\label{lemma_two_product}
    \begin{equation}
        \frac{1}{C}K_n(t-t^\prime)\leq\mathbb E[M_n(t)M_n(t^\prime)] \leq CK_n(t-t^\prime).
        \label{two_product_expectation_estimate}
    \end{equation}
\end{lemma}
\begin{proof}This fact is well known (cf. \cite{Kah85b, Fan89, Fan91}). For the convenience of the reader, we sketch the proof.

Firstly, for each $k\geq1$ such that $\ell_k\leq\frac{1}{2}$, we have
\begin{equation*}
\begin{split}
\mathbb{E}\left[P_k(t)P_k(t^\prime)\right] &= \frac{1 - 2\ell_k + \mathbf{1}_{(0,\ell_k)}\ast\mathbf{1}_{(0,\ell_k)}(t-t^\prime}{(1 - \ell_k)^2} 
= \frac{1 - 2\ell_k + (\ell_k-|t-t^\prime|)_+}{(1 - \ell_k)^2}.
\end{split}
\end{equation*}

Then we finish the estimation by using the elementary inequality: there is $c>0$ such that for any $x\in(-\frac{1}{4},\frac{1}{4})$,
    $
    e^{-x-cx^2} \leq 1-x \leq e^{-x},
    $
and the fact $\sum_{n=1}^\infty\ell_n^2<\infty$.
\end{proof}

Now we are going to prove \eqref{ultimate_goal}. Let us first look at $F_n(\hat{t},t_{d+1})$ involved in $\ast^dM_n(t)$. As \(\sum_{n=1}^{\infty} \ell_n = \infty\) we have \(K(0) = \infty\). This suggests that if two of points \(t_1, \ldots, t_d, t_{d+1}\) in the definition of $F_n(\hat{t},t_{d+1})$ are close, the function $F_n(\hat{t},t_{d+1})$ takes large values. In order to control $\ast^dM_n(t)$, we decompose the integral domain $\mathbb{T}^d$ in the following way: $\mathbb{T}^d = A_\delta(t)\cup A_\delta(t)^c$
for given $t\in\mathbb{T}$ where
\begin{equation*}
A_\delta(t) = \bigcup_{1\leq i\neq j\leq d+1} A_{\delta,i,j}(t)\quad \text{with}\quad A_{\delta,i,j}(t):=\{(t_1,\cdots, t_d)\in \mathbb T^{d}: \|t_i-t_j\|\leq \delta\}.
\end{equation*}
Notice that we can write
\begin{equation}
\begin{aligned}
    &\mathbb{E} \bigl[ \|\ast^d M_n(t) - \ast^d M_m(t)\|_1^{\frac{1}{2}} \bigr]
    = \mathbb{E} \biggl[ \biggl( \int_{\mathbb{T}} \biggl| \int_{\mathbb{T}^d} [F_n(\hat{t}, t) - F_m(\hat{t}, t)] \, \mathrm{d}\hat{t} \biggr| \, \mathrm{d}t \biggr)^{\frac{1}{2}} \biggr]. 
\end{aligned} \label{decomposition-ultimate-goal}
\end{equation}
The last term is bounded by
\begin{equation*}
         \mathbb{E} \biggl[ \biggl( \int_{\mathbb{T}} \biggl| \int_{A_\delta(t)^c} [F_n(\hat{t}, t) - F_m(\hat{t}, t)] \, \mathrm{d}\hat{t} \biggr| \, \mathrm{d}t \biggr)^{\frac{1}{2}} \biggr] + L_n^{(\delta)} + L_m^{(\delta)}
\end{equation*}
where 
$$
L_n^{(\delta)} := \mathbb{E} \biggl[ \biggl( \int_{\mathbb{T}} \int_{A_\delta(t)} F_n(\hat{t}, t) \, \mathrm{d}\hat{t} \, \mathrm{d}t \biggr)^{\frac{1}{2}} \biggr].
$$
This estimation is a simple consequence of the sub-addtivity $\sqrt{a+b}\leq\sqrt{a}+\sqrt{b}\,\,(\forall\, a,b\geq0)$ and the H\"older inequality.

Hence we will completes the proof of Theorem~\ref{multiple_convolution_result} by proving the following two limits.

\begin{proposition}\label{energy-integral} Under the assumptions of Theorem~\ref{multiple_convolution_result},
\begin{equation*}
\lim_{\delta\rightarrow0+}\sup_{n\geq1}L_n^{(\delta)}=0,\text{ i.e. }\lim_{\delta\rightarrow0+}\sup_{n\geq1}\mathbb{E} \biggl[ \biggl( \int_{\mathbb{T}} \int_{A_\delta(t)} F_n(\hat{t}, t) \, \mathrm{d}\hat{t} \, \mathrm{d}t \biggr)^{\frac{1}{2}} \biggr] = 0. 
\end{equation*}
\end{proposition}
\begin{proposition}\label{L1conver}Under the assumptions of Theorem~\ref{multiple_convolution_result}
$$
\lim_{n,m\rightarrow0}\mathbb{E} \biggl[ \biggl( \int_{\mathbb{T}} \biggl| \int_{\mathbb{T}^d} [F_n(\hat{t}, t) - F_m(\hat{t}, t)] \, \mathrm{d}\hat{t} \biggr| \, \mathrm{d}t \biggr)^{\frac{1}{2}} \biggr] = 0.
$$
\end{proposition}

\subsection{Proof of Proposition~\ref{energy-integral}}
The proof of the proposition is rather easy. It depends on the following result contained in \cite{Fan95}:
Under the assumption \ref{B}, 
the martingale \( \mu_n (\mathbb{T}) \) converges in \( L^p \) for all \( p \geq 1 \).

        Note that $(t,\hat{t})\in \mathbb T \times A_\delta(t)$ means that there are $i\neq j$ such that $\|t_i-t_j\| \leq \delta$. The number of choices of such pair $(i,j)$ is $d(d+1)$. Then using one more the above mentioned sub-additivity inequality, we get
        \begin{equation}
            L_n^{(\delta)} \leq \sum_{1\leq i\neq j\leq d+1} L_{n,i,j}^{(\delta)} \label{L_n_i_j}
        \end{equation}
        where 
        $$
        L_{n,i,j}^{(\delta)} = \mathbb{E} \biggl[ \biggl( \int_\mathbb{T}\int_{A_{\delta,i,j}(t)} F(\hat{t},t)\, \mathrm{d}\hat{t} \, \mathrm{d}t \biggr)^{\frac{1}{2}} \biggr].
        $$

To estimate $L_{n,i,j}^{(\delta)}$, we distinguish two cases: $1\leq i\neq j\leq d$ and $1\leq i < j =d+1$. Assume first $1\leq i\neq j\leq d$.
Notice that by the translation invariant of Lebesgue measure on $\mathbb{T}$ we have
$$
\int_{\mathbb T} M_n(t-\sum_{i=1}^dt_i) \mathrm{d}t =\int_\mathbb{T}M_n(t_k)dt_k= \mu_n(\mathbb{T}).
$$
This together with Fubini's theorem and the symmetry of $t_1,\cdots,t_d$ in $F(\hat{t},t)$, gives us
$$
L_{n,i,j}^{(\delta)}=L_{n,1,2}^{(\delta)}=\mathbb E\left[\left(\int_{\|t_1-t_2\|\leq \delta}M_n(t_1)M_n(t_2)\mathrm{d}t_1\mathrm{d}t_2 \right)^\frac{1}{2}\mu_n(\mathbb{T})^{\frac{d-1}{2}}\right].
$$

Assume now $1\leq i< j= d+1$. Making the change of variable $s_k=t_k$ $(1\leq k\leq d)$, and $s_{d+1}=t-\sum_{n=1}^dt_n$, which is a measure-preserving transformation on $\mathbb{T}^d$ with respect to $\|\cdot\|$, we get
$$
L_{n,i,d+1}^{(\delta)} = \mathbb E\left[\left(\int_{\|s_i-s_{d+1}\|\leq \delta}M_n(s_i)M_n(s_{d+1})\mathrm{d}s_i\mathrm{d}s_{d+1} \right)^\frac{1}{2}\mu_n(\mathbb{T})^{\frac{d-1}{2}}\right] = L_{n,1,2}^{(\delta)}.
$$

Hence, by Cauchy-Schwartz inequality, from \eqref{L_n_i_j}, we get
$$
L_n^{(\delta)}\leq d(d+1)  \left(\mathbb E\left[\mu_n(\mathbb{T})^{d-1}\right]\right)^\frac{1}{2}\left(\mathbb E \left[\int_{\|t_1-t_2\|\leq \delta}M_n(t_1)M_n(t_2)\mathrm{d}t_1\mathrm{d}t_2\right]\right)^\frac{1}{2}.
$$

        By Fan's result stated above, $\mathbb{E}\left[\mu_n(\mathbb{T})^{d-1}\right]\leq C^2$ for some $C>0$ and any $n\geq1$. Then by Fubini's theorem and Lemma~\ref{two_product_expectation_estimate}, we get
        \begin{align*}
                      L_n^{(\delta)} \leq d(d+1)C
                      \left(\int_{\|t_1-t_2\|\leq \delta}K(t_1-t_2)\mathrm{d}t_1\mathrm{d}t_2\right)^\frac{1}{2}.
        \end{align*}
       This completes the proof of proposition~\ref{energy-integral}.

\subsection{Proof of Proposition~\ref{L1conver}} 
\subsubsection{Reduction}We first get rid of the exponent $\frac{1}{2}$ and prove a stronger result than needed:
$$
\lim_{n,m\rightarrow\infty} \int_{\mathbb{T}} \mathbb{E}\biggl[\left| \int_{A_\delta(t)^c} \left(F_n(\hat{t}, t) - F_m(\hat{t}, t) \right)\, \mathrm{d}\hat{t} \right|\biggr] \, \mathrm{d}t=0.
$$
This is because of the H\"older inequality and Fubini's theorem. First we observe that the above expectation is indeed independent of $t$.
\begin{lemma}\label{lem-independent-of-t}
    For any $t\in\mathbb{T}$, one has 
    $$
    \mathbb{E}\biggl[ \left|\int_{A_\delta(t)^c} \left(F_n(\hat{t}, t) - F_m(\hat{t}, t)\right) \, \mathrm{d}\hat{t} \right|\biggr] = \mathbb{E}\biggl[ \left|\int_{A_\delta(0)^c} \left(F_n(\hat{t}, 0) - F_m(\hat{t}, 0)\right) \, \mathrm{d}\hat{t} \right|\biggr].
    $$
\end{lemma}
to the integral $\int_{A_\delta(t)^c} \left(F_n(\hat{t}, t) - F_m(\hat{t}, t)\right) \, \mathrm{d}\hat{t}$. Note that 
\begin{proof}Firstly, for the integral inside the expectation, we introduce a substitution that uniformly shifts the variables, to remove the dependence on $t$ of variable $t_{d+1} = t - \sum_{k=1}^{n} t_k$.
Let $s_i=t_i-\frac{t}{d+1}$ $(i=1,\cdots,n)$ then
$$
t_{d+1} = t-\sum_{i=1}^\mathrm{d}t_i = \frac{t}{d+1}-\sum_{i=1}^ds_i.
$$
    The set $A_\delta(t)^c$ consists of the points $(t_1,\cdots,t_d)$ such that any two of $t_1,\cdots,t_d,t_{d+1}$ has distance at least $\delta$. Under the change of variable the conditions \( \|t_i - t_j\| = \|s_i - s_j\| \geq \delta \) for \( 1\leq i \neq j\leq d \) are unaffected by the shift. The conditions involving $t_{d+1}$ transform as:
\[
\left\|t_j - (t - \sum_{i=1}^d t_i) \right\| = \left\| s_j - \sum_{i=1}^d s_i  \right\|\geq \delta.
\]
Hence the set $A_\delta(t)^c$ becomes 
    $$
    A_\delta(0)^c = \{\hat{s}=(s_1,\cdots,s_d)\in\mathbb{T}^d: \|s_i-s_j\|\geq\delta, 1\leq i\neq j\leq d+1\}
    $$
    where $s_{d+1} = -\sum_{i=1}^ds_i$.

By definition of $F_n$ in \eqref{def_F_n}, one has
        \begin{align*}
            F_n(\hat{t}, t)
            =\left(\prod_{i=1}^dM_n(s_i+\frac{t}{d+1})\right)M_n(\frac{t}{d+1}-\sum_{i=1}^ds_i)= F_n(\hat{s}+\frac{t}{d+1}, \frac{t}{d+1})
        \end{align*}
where $\hat{s}+\frac{t}{d+1} = (s_1+\frac{t}{d+1},\cdots,s_d+\frac{t}{d+1})$.

   Therefore the integral becomes 
    \begin{equation}
           \int_{A_\delta(t)^c} \left(F_n(\hat{t}, t) - F_m(\hat{t}, t)\right) \, \mathrm{d}\hat{t} = \int_{A_\delta(0)^c} \left(F_n(\hat{s}+\frac{t}{d+1}, \frac{t}{d+1}) - F_m(\hat{s}+\frac{t}{d+1}, \frac{t}{d+1})\right) \, d\hat{s}.\label{change-variable}
    \end{equation}
 
    By the definition of $P_k(t)=\frac{1-\textbf{1}_{(0,\ell_k)}(t-\omega_k)}{1-\ell_k}$, there is a bounded measurable function  $g$ on $\mathbb{T}$
    such that 
    $$ P_k(t) = g(\omega_k-t).$$

Hence
$$
F_n(\hat{s}+\frac{t}{d+1}, \frac{t}{d+1})=\prod_{k=1}^n\left(g(\omega_k-\frac{t}{d+1}+\sum_{i=1}^ds_i)\prod_{i=1}^{d}g(\omega_k-\frac{t}{d+1}-s_i)\right).
$$
    For any $k\geq 1$, since $\omega_k$ follows uniform distribution on $ \mathbb T $, its distribution is invariant under translation, and thus
    $$
     F_n(\hat{s}+\frac{t}{d+1},\frac{t}{d+1}) \overset{d}{=} F_n(\hat{s},0)
    $$
  which together with \eqref{change-variable} implies
   \begin{align*}
           \mathbb{E}\biggl[ \left|\int_{A_\delta(t)^c} \left(F_n(\hat{t}, t) - F_m(\hat{t}, t)\right) \, \mathrm{d}\hat{t} \right|\biggr] 
           =\mathbb E\left[\left|\int_{A_\delta(0)^c}  \left(F_n(\hat{s}, 0) - F_m(\hat{s}, 0)\right) \, d\hat{s}\right|\right].
   \end{align*}
\end{proof}

Lemma~\ref{lem-independent-of-t} tells us that what we have to prove is
$$
 J_n:=\int_{A_\delta(0)^c} F_n(\hat{t}, 0)  \, \mathrm{d}\hat{t}\, \text{ is a Cauchy sequence } L^1(\mathbb P),
$$
because this implies Proposition~\ref{L1conver}. We can prove a little bit more:
\begin{proposition}
Under the assumptions of Theorem~\ref{multiple_convolution_result},
    $J_{n}$ is a Cauchy sequence in $L^2(\mathbb P)$. \label{L2-cauchy}
\end{proposition}
The proof of Proposition~\ref{L2-cauchy} follows an elegant method in \cite{Ber17} designed to deal with Gaussian multiplicative chaos and the key step is to show $J_{n}$ is $L^2$-bounded.

\subsubsection{$L^2$-Boundedness of $J_{n}$}
\begin{lemma} Under the assumptions of Theorem~\ref{multiple_convolution_result}, 
        $J_{n}$ is $L^2$-bounded.
    \label{L^2bounded}
\end{lemma} 
\begin{proof}
By Fubini's theorem, we have\begin{equation}
    \mathbb{E}\left[J_n^2\right] = \int_{A_\delta(0)^c}\int_{A_\delta(0)^c}\mathbb{E}\left[F_n(\hat{t},0)F_n(\hat{t}^\prime,0)\right]\mathrm{d}\hat{t}\mathrm{d}\hat{t}^\prime. \label{Fubini-of-J_n}
\end{equation}

    Note that for any $\delta > 0$, $\mathbb T$ can be divided into at most $K\triangleq ([\frac{2}{\delta}]+1)$ sub-intervals denoted $B_1,\cdots, B_k$, such that the length of each interval is at most $\frac{\delta}{2}$.

    Hence, by the definition of $A_\delta(0)$, for any $(\theta_1,\cdots,\theta_d)\in A_\delta(0)$, there are $d+1$ disjoint intervals denoted $B_{j_1},\cdots, B_{j_{d+1}}$ such that $t_i\in B_{j_{i}}$ and $dist(B_{j_i},B_{j_k}) \geq \frac{\delta}{2}$ for any $i\neq k \in \{1,\cdots,d+1\}$ where $t_{d+1} =-\sum\limits_{i=1}^{d}t_i.$ Note that the number of the choices of $B_{j_1},\cdots, B_{j_{d+1}}$ are at most $\binom{K}{d+1}$, which, together with the above mentioned sub-additivity and Fubini's theorem, implies
\begin{equation}
\begin{split}
J_n^\frac{1}{2} &\leq \sum_{B_{j_1}, \cdots , B_{j_{d}}} \left(\mathbb E\left[(\int_{B_{j_1} \times \cdots \times B_{j_{d}}}F_{n}(\hat{t},0)\text{1}_{B_{d+1}}(t_{d+1})\mathrm{d}\hat{t}\,)^2\right]\right)^\frac{1}{2}\\
= \sum_{B_{j_1}, \cdots , B_{j_{d}}} &\left(\int_{(B_{j_1} \times\cdots \times B_{j_{d}})^2}\mathbb E\left[F_n(\hat{t},0)F_n(\hat{t}^\prime,0)\right]\textbf{1}_{B_{j_{d+1}}}(t_{d+1})\textbf{1}_{B_{j_{d+1}}}(t_{d+1}^\prime)\mathrm{d}\hat{t}\mathrm{d}\hat{t}^\prime\right)^\frac{1}{2}. 
\end{split}\label{expension-of-J_n}
\end{equation}

Since $\{\omega_k\}_{k}$ is an i.i.d sequence, we can decompose
$$\mathbb E\left[F_n(\hat{t},0)F_n(\hat{t}^\prime,0)\right] = \prod_{k=1}^n\mathbb{E}\left[\prod_{i=1}^{d+1}P_k(t_i)P_k(t_i^\prime)\right].$$

By the fact $\sum_{k=1}^\infty\ell_k^2<\infty$, there is $N\geq1$, only depending on $\delta$, such that 
\( \ell_k < \delta/2 \) for any \(k\geq N\). Given that each \( t_i, t_i^\prime \in B_{j_i} \), and the \( B_{j_i} \) are pairwise separated by a distance of at least \( \delta/2 \), it follows that for \( \ell_k < \delta/2 \), then the sets \( (t_i - \ell_k, t_i) \cup (t_i^\prime - \ell_k, t_i^\prime) \) for distinct \( i \) do not intersect. Hence, if $k>N$,
\begin{equation}\label{calculation-of-de}
\begin{split}
&\mathbb{P}\left[\omega_k \notin \bigcup_{i=1}^{d+1} (t_i - \ell_k, t_i) \cup (t_i^\prime - \ell_k, t_i^\prime)\right]
= 1 - 2(d+1)\ell_k + \sum_{k=1}^{d+1} (\ell_k - |t_i - t_i^\prime|)_+
\end{split}
\end{equation}
where $|A|$ means the Lebesgue measure of measurable set $A \subset\mathbb{T}$.

Therefore there is a $C>0$ which is only depending on $N$ such that 
$$
\mathbb E\left[F_n(\hat{t},0)F_n(\hat{t}^\prime,0)\right]\leq C\prod_{k=1}^n\frac{1 - 2(d+1)\ell_k +\sum_{i=1}^{d+1} (\ell_k - |t_i - t_i^\prime|)_+}{(1-\ell_k)^{2(d+1)}}.
$$
Following the exactly same argument of Lemma~\ref{lemma_two_product}, one can deduce that
\begin{equation}
    \mathbb E\left[F_n(\hat{t},0)F_n(\hat{t}^\prime,0)\right] \leq C^\prime K(t_1-t_1^\prime) K(t_2-t_2^\prime) \cdots K(t_{d+1}-t_{d+1}^\prime) \label{dominated-cov-condition}
\end{equation}
then the desired result follows from \eqref{expension-of-J_n} and the assumption \eqref{L1condition_RC}.
\end{proof}

\subsubsection{$L^2$-Convergence of $J_{n}$}
Observe that 
\[
\mathbb{E}\left[(J_n - J_m)^2\right] = \mathbb{E}\left[ J_n^2\right] + \mathbb{E} \left[J_m^2\right] - 2 \mathbb{E}\left[J_n J_m\right].
\]
The $L^2$-convergence of $J_n$ will follows if we can prove that
there exists a constant \( A > 0 \) such that
\[
\limsup_{n \to \infty} \mathbb{E} \left[J_n^2\right] \leq A \leq \liminf_{n,m \to \infty} \mathbb{E}\left[J_n J_m\right].
\]
Actually, we will prove these inequalities with 
$$
A = \int_{A_\delta(0)^c}\int_{A_\delta(0)^c}  H(\hat{t}, \hat{t}^\prime)\mathrm{d}\hat{t}\mathrm{d}\hat{t}^\prime
$$
where 
$$
H(\hat{t}, \hat{t}^\prime) :=  \prod_{k=1}^\infty \mathbb{E} \biggl[ \prod_{i=1}^{d+1} P_k(t_i) P_k(t_i^\prime) \biggr].
$$
\begin{lemma}
    Under the assumptions of Theorem~\ref{multiple_convolution_result},
    $$
    \limsup_{n \to \infty} \mathbb{E} \left[J_n^2\right] \leq A\qquad \text{and}\qquad\liminf_{n,m \to \infty} \mathbb{E}[J_n J_m]\geq A.
    $$
    \label{lim<part}
\end{lemma}

The basic idea of proof of Lemma~\ref{<eta_part} is as follows. In \eqref{Fubini-of-J_n}, observe that when $t_i$ and $t_i^\prime$ are separated, the integrand in $\mathbb{E}\left[J_n^2\right]$ will converges uniformly, while when $t_i$ and $t_i^\prime$ lie close to each other, the integral can be controlled arbitrarily small due to \eqref{L1condition_RC} and \eqref{dominated-cov-condition}.

For any $\eta>0$, let $S_\eta:= \{(\hat{t},\hat{t}^\prime)\in A_\delta(0)^c\times A_\delta(0)^c: \|t_i-t_i^\prime\|\leq \eta,\, \text{for some }i=1,\cdots,d+1\}$, it is worthy to recall that $t_{d+1} = -\sum_{k=1}^dt_k$ when $t=0$. By \eqref{dominated-cov-condition}, we know that $\mathbb E\left[F_n(\hat{t},0)F_n(\hat{t}^\prime,0)\right]$ is dominated by $K(t_1-t_1^\prime) K(t_2-t_2^\prime) \cdots K(t_{d+1}-t_{d+1}^\prime)$ which is integrable due to \eqref{L1condition_RC}, then by Lebesgue dominated convergence theorem, we have
\begin{equation}
    \lim_{\eta\rightarrow 0}\sup_n\int_{(A_\delta(0)^c)^2\cap S_\eta}\mathbb E\left[F_n(\hat{t},0)F_n(\hat{t}^\prime,0)\right]\mathrm{d}\hat{t}\mathrm{d}\hat{t}^\prime = 0. \label{<eta_part}
\end{equation}

\begin{lemma}\label{uniform-convergence}
    Under the assumptions of Theorem~\ref{multiple_convolution_result},
    when $(\hat{t},\hat{t}^\prime)\in(A_\delta(0)^c)^2\cap S_\eta^c$, the following infinite product converges
    \[
H(\hat{t}, \hat{t}^\prime) :=  \prod_{k=1}^\infty \mathbb{E} \biggl[ \prod_{i=1}^{d+1} P_k(t_i) P_k(t_i^\prime) \biggr].  
\]
    and the following limit holds
    \[\lim_{n\rightarrow\infty}\sup_{(\hat{t},\hat{t}^\prime)\in(A_\delta(0)^c)^2\cap S_\eta^c}\mathbb E\left[F_n(\hat{t},0)F_n(\hat{t}^\prime,0)\right] \rightarrow  H(\hat{t}, \hat{t}^\prime) \label{uniform-convergence}.
    \]
\end{lemma}
\begin{proof}
    We may assume that $\eta<\frac{\delta}{2}$, and for the given $\eta$, by the fact $\sum_{k=1}^\infty\ell_k^2< \infty$, there is $M\geq1$ such that 
\( \ell_k\leq\eta \leq \frac{\delta}{2} \) for any \(k\geq M\).
Over the region $A_\delta(0)^c$ and $S_\eta^c$, for $k\geq M$, we not only have $\|t_i-t_j\|\geq\frac{\delta}{2}\geq\ell_k$, but also have $\|t_i-t_i^\prime\|\geq\eta\geq \ell_k$ which implies that all of $\ell_k$-length intervals are disjoint to each other, and thus
\begin{equation}\label{calculation-of-de-eta}
\begin{split}
\mathbb{P}\left[\omega_k \notin \bigcup_{i=1}^{d+1} (t_i - \ell_k, t_i) \cup (t_i^\prime - \ell_k, t_i^\prime)\right] 
=& 1 - 2(d+1)\ell_k
\end{split}
\end{equation}
which is independent of $t$ and $t^\prime$.

Hence for $k\geq M$,
$$\mathbb{E}\left[\prod_{i=1}^{d+1}P_k(t_i)P_k(t_i^\prime)\right] =\frac{1-2(d+1)\ell_k}{\left(1-\ell_k\right)^{2(d+1)}}$$ which implies $H(\hat{t},\hat{t}^\prime)$ exists and that
$$
\mathbb E\left[F_n(\hat{t},0)F_n(\hat{t}^\prime,0)\right] = \prod_{k=1}^n\mathbb{E}\left[\prod_{i=1}^{d+1}P_k(t_i)P_k(t_i^\prime)\right] \longrightarrow H(\hat{t},\hat{t}^\prime) 
$$
holds uniformly over $A_\delta(0)^c\times A_\delta(0)^c$ and $S_\eta^c$ as $n$ goes to $\infty$.
\end{proof}

Finally, we are ready to prove the Lemma~\ref{lim<part}.

\begin{proof}[Proof Lemma~\ref{lim<part}]
    For the first inequality, we split $\mathbb E\left[J_n^2\right]$ into two parts, for $0<\eta\leq\frac{\delta}{2}$,
\begin{equation}
\begin{split}
        &\mathbb E\left[J_n^2\right] =  \int_{A_\delta(0)^c}\int_{A_\delta(0)^c}\mathbb{E}\left[F_n(\hat{t},0)F_n(\hat{t}^\prime,0)\right]\mathrm{d}\hat{t}\mathrm{d}\hat{t}^\prime\\=&\left(\int_{(A_\delta(0)^c)^2\cap S_\eta}+ \int_{(A_\delta(0)^c)^2\cap (S_\eta)^c}\right)\mathbb E\left[F_n(\hat{t},0)F_n(\hat{t}^\prime,0)\right]\mathrm{d}\hat{t}\mathrm{d}\hat{t}^\prime.
\end{split}\label{squre-est}
\end{equation}
We denote the first and second integral as $I_1$ and $I_2$ respectively. We first let $n$ go to infinity and then $\eta$ go to 0 in order on both sides of \eqref{squre-est}.

For $I_1$, by \eqref{<eta_part}, one has 
$
\lim_{\eta\rightarrow0}\sup_n I_1 =0.
$

For $I_2$, by Lemma~\ref{uniform-convergence}, one has 
$$
\lim_{\eta\rightarrow0}\lim_{n\rightarrow\infty}I_2 = \int_{(A_\delta(0)^c)^2}H(\hat{t}, \hat{t}^\prime)\mathrm{d}\hat{t}\mathrm{d}\hat{t}^\prime.
$$
This completes the proof of the first inequality.

For the other inequality, by Fubini's theorem and the integrand is non-negative we have
\begin{equation*}
\begin{split}
        \mathbb E\left[J_nJ_m\right] = \int_{(A_\delta(0)^c)^2}\mathbb E\left[F_n(\hat{t},0)F_m(\hat{t}^\prime,0)\right]\mathrm{d}\hat{t}\mathrm{d}\hat{t}^\prime&\geq 
        \int_{(A_\delta(0)^c)^2\cap S_\eta^c}\mathbb E\left[F_n(\hat{t},0)F_m(\hat{t}^\prime,0)\right]\mathrm{d}\hat{t}\mathrm{d}\hat{t}^\prime.
\end{split}
\end{equation*}
Following exactly the same argument in Lemma~\ref{uniform-convergence}, we have $\mathbb E\left[F_n(\hat{t},0)F_m(\hat{t}^\prime,0)\right]$ is also uniformly convergent to $H(\hat{t},\hat{t}^\prime)$ on $\left(A_\delta(0)^c\times A_\delta(0)^c\right)\cap S_\eta^c$.

Hence, let $n,m$ go to infinity and then $\eta$ go to $0$, we get
    $$
    \liminf_{n \to \infty} \mathbb{E}[J_nJ_m]\geq\int_{A_\delta(0)^c} \int_{A_\delta(0)^c} H(\hat{t}, \hat{t}^\prime)\mathrm{d}\hat{t}\mathrm{d}\hat{t}^\prime.
    $$
\end{proof}

\section*{Acknowledgments}  

Special thanks to Professor Fan Aihua for his invaluable guidance and suggestions during the preparation of this paper.

\newcommand{\etalchar}[1]{$^{#1}$}

\vspace{1.5em} 
\noindent
\begin{minipage}{\textwidth}
    \small 
    \begin{flushleft}
        \textsc{LAREMA, UMR 6093 CNRS, Université d’Angers,   Angers cedex, France},\\
        \textit{Email address:} \texttt{mingjie.tan@univ-angers.fr}
    \end{flushleft}
\end{minipage}

\end{document}